\documentclass[11pt,reqno,twoside]{amsart}

\usepackage{amsfonts,amsmath,amssymb}
\usepackage{mathrsfs,mathtools}
\usepackage{enumerate}
\usepackage{hyperref}
\usepackage{esint}
\usepackage{graphicx}
\usepackage{bm}
\usepackage{commath}
\usepackage{esint}
\DeclareMathAlphabet{\mathpzc}{OT1}{pzc}{m}{it}

\usepackage{placeins} 
\usepackage{flafter} 

\usepackage{tikz}

\usepackage[textsize=small]{todonotes}
\numberwithin{equation}{section}

\usepackage[dvips,bottom=1.4in,right=1in,top=1in, left=1in]{geometry}

\makeatletter
\def\eqnarray{\stepcounter{equation}\let\@currentlabel=\theequation
\global\@eqnswtrue
\tabskip\@centering\let\\=\@eqncr
$$\halign to \displaywidth\bgroup\hfil\global\@eqcnt\z@
  $\displaystyle\tabskip\z@{##}$&\global\@eqcnt\@ne
  \hfil$\displaystyle{{}##{}}$\hfil
  &\global\@eqcnt\tw@ $\displaystyle{##}$\hfil
  \tabskip\@centering&\llap{##}\tabskip\z@\cr}

\def\endeqnarray{\@@eqncr\egroup
      \global\advance\c@equation\m@ne$$\global\@ignoretrue}

\newtheorem{theorem}{Theorem}[section]

\newtheorem{corollary}[theorem]{Corollary}
\newtheorem{definition}[theorem]{Definition}

\newtheorem{proposition}[theorem]{Proposition}
\newtheorem{assumption}[theorem]{Assumption}
\newtheorem{remark}[theorem]{Remark}
\numberwithin{equation}{section}

\def\RR{{\mathbb{R}}}
\def\NN{{\mathbb{N}}}

\def\Om{\Omega}
\def\om{\omega}

\def\om{\omega}

\def\ga{\gamma}

\def\bOm{\overline{\Om}}
\def\pOm{\partial\Omega}

\title[Optimal Control of Fractional PDEs with State Constraints: Parabolic Case]{Moreau-Yosida Regularization for Optimal Control of Fractional PDEs with State Constraints: Parabolic Case 
\\ 
{\it \tiny (Dedicated to Prof. Enrique Zuazua on his 60th Birthday)}}

\author{Harbir Antil, Thomas S. Brown, Deepanshu Verma, Mahamadi Warma}
\address{Department of Mathematical Sciences and the Center for Mathematics and Artificial Intelligence (CMAI), 
		 George Mason University, Fairfax, VA 22030, USA.}
\email{hantil@gmu.edu, tbrown62@gmu.edu, dverma2@gmu.edu, mwarma@gmu.edu}

\thanks{The first three authors are partially supported by NSF grants DMS-1818772, DMS-1913004, 
the Air Force Office of Scientific Research (AFOSR) under Award NO: FA9550-19-1-0036, and 
Department of Navy, Naval PostGraduate School under Award NO: N00244-20-1-0005. 
The fourth author is partially supported by the AFOSR under Award NO:  FA9550-18-1-0242.
}

\keywords{Fractional parabolic PDE, optimal control, state and control constraints,
measure valued datum, Moreau-Yosida regularization, convergence with rate.}

\subjclass[2010]{
49J20,  	
49K20,   
35S15,  	
65R20  	
}

\begin{document}

\begin{abstract}
This paper considers optimal control of fractional parabolic PDEs with both state and 
control constraints. The key challenge is how to handle the state constraints. 
Similarly, to the elliptic case, in this paper, we establish several new 
mathematical tools in the parabolic setting that are of wider interest. For example, existence 
of solution to the fractional parabolic equation with measure data on the right-hand-side. 
We employ the Moreau-Yosida regularization to handle the state constraints. We establish 
convergence, with rate, of the regularized optimal control problem to the original one. 
Numerical experiments confirm what we have proven theoretically. 
\end{abstract}

\maketitle

\section{Introduction}

Fractional diffusion models have recently received a tremendous amount of attention. They 
have shown a remarkable flexibility in capturing anomalous behavior. Recently in 
\cite{weiss2020fractional}, the authors have derived the fractional Helmholtz equations 
using the first principle arguments combined with a constitutive relationship. This article 
also shows a qualitative match between the real data and the numerical experiments. It is
natural to consider the optimal control or inverse problems with such partial differential 
equations (PDEs) as constraints and has also motivated the current study.

The main goal of this paper is to consider an optimal control problem with fractional parabolic 
PDEs as constraints. We denote the fractional exponent by $s$ which lies strictly between 0 and 1. 
The key novelty stems from the state constraints. Optimal control of fractional
PDEs have recently received a significant amount of attention. We refer to 
\cite{HAntil_DVerma_MWarma_2019b,antil2019moreau} for optimal control of fractional elliptic 
PDEs with state constraints. Moreover, see \cite{HAntil_MWarma_2020a} for the optimal 
control of semilinear fractional PDEs with control constraints, see also \cite{MDElia_MGunzburger_2014a} 
for the linear case. We also refer to 
\cite{HAntil_RKhatri_MWarma_2019a,HAntil_DVerma_MWarma_2020a} for the exterior optimal control 
of fractional elliptic and parabolic PDEs with control constraints.

We emphasize that the optimal control of parabolic PDEs with state constraints is not new, 
see for instance \cite{ECasas_1997a} for the classical case $s = 1$. Similarly to our observation 
in the elliptic case \cite{HAntil_DVerma_MWarma_2019b}, we emphasize that almost none of the existing 
works can be directly applied to our fractional setting. For instance, we need finer integrability 
conditions on the data to show the continuity of solution to the fractional PDEs, and we need to create the notion 
of very-weak solution to fractional parabolic PDEs with measure-valued datum.

Next we present the statements of our main results on fractional PDEs while delaying proofs to the body of the paper.  We do this in order to show the novelties of this work and clearly distinguish it from previous similar, yet different, papers.  The novelties for the optimal control problem are described after these results. For $\Om \subset \RR^N$ $(N\ge 1)$, a bounded open set with boundary $\pOm$, we wish to study the following fractional parabolic optimal control 
problem: 
\begin{subequations}\label{eq:dcp}
 \begin{equation}\label{eq:Jd}
    \min_{(u,z)\in (U,Z)} J(u,z) 
 \end{equation}
 subject to the fractional parabolic PDE: Find $u \in U$ solving 
 \begin{equation}\label{eq:Sd}
 \begin{cases}
    \partial_t u + (-\Delta)^s u &= z \quad \mbox{in } Q:= (0,T) \times \Om , \\
                u &= 0 \quad \mbox{in } \Sigma := (0,T) \times (\RR^N\setminus \Om)  , \\
       u(0,\cdot) &= 0 \quad \mbox{in } \Omega  , 
 \end{cases}                
 \end{equation} 
with the state constraints 
 \begin{equation}\label{eq:Ud}
    u|_{Q} \in \mathcal{K} := \left\{ w \in C(\overline{Q}) \ : \ w(x,t) \le u_b(x,t) , 
        \quad \forall (x,t) \in \overline{Q}  \right\} 
 \end{equation}
 and control constraint imposed for 
 \begin{equation}\label{eq:Zd}
        z \in Z_{ad} \subset L^r((0,T),L^p(\Om)).
\end{equation}
\end{subequations}
Here, and subsequently, $C(\overline{Q})$ is the space of continuous functions in $\overline{Q}$ and 
$u_b \in C(\overline{Q})$.  Furthermore,  $Z_{ad}$ is a non-empty, closed and convex set, and the real numbers $p$ and
$r$ fulfill
\begin{equation}
1 > \frac{N}{2ps} + \frac{1}{r} \, . \label{cond-p1}  
\end{equation}
The precise definition of the function spaces $U$ and $Z$ will be given in the forthcoming 
sections. We note that the control problem with nonlocal PDE constraint in \eqref{eq:Sd} may look similar to \cite{HAntil_DVerma_MWarma_2020a} at first glance, however, the two problems are significantly different. First of all, the control is taken as a source in this paper rather than in the exterior of the domain in \cite{HAntil_DVerma_MWarma_2020a}. Secondly, in the current paper, we consider both the control and the state constraints rather than just the control constraints in \cite{HAntil_DVerma_MWarma_2020a}.

Our first main result, stated below, involves the boundedness of $u$. 
    \begin{theorem}[\bf $u$ is bounded]
    \label{thm:bdd}
    Let $\Omega\subset\RR^N$ ($N\ge 1$) be an arbitrary bounded open set.
        Let $z \in L^r((0,T);L^p(\Om))$ with $p$ and $r$ fulfilling 
        \eqref{cond-p1}. Then every weak solution $u \in \mathbb{U}_0$ to 
        \eqref{eq:Sd} belongs to $L^\infty(Q)$ and there is a constant $C>0$ such that
        \begin{equation}\label{eq:Linfbnd}
            \|u\|_{L^\infty(Q)} \le C \|z\|_{L^r((0,T);L^p(\Om))} . 
        \end{equation}
    \end{theorem}
We remark that a similar result has been recently shown in \cite[Theorem~29 and Corollary~3]{TLeonori_IPeral_APrimo_FSoria_2015a}. First, our proof (see Section \ref{s:statecont}) is much simpler, it follows directly by using semigroup arguments.  Secondly, we were not able to follow the arguments given in the proof of  \cite[Theorem~29]{TLeonori_IPeral_APrimo_FSoria_2015a}. We think that the proof maybe contain some non obvious misprints.

From Theorem \ref{thm:bdd}, we get the continuity of $u$ as a corollary.
\begin{corollary}[\bf $u$ is continuous] \label{cor:ucont}
Let $\Omega\subset\RR^N$ ($N\ge 1$) be a bounded open set with a Lipschitz continuous boundary and has the exterior ball condition. Let  $z \in L^r((0,T);L^p(\Om))$ with $p$ and $r$ fulfilling
        \eqref{cond-p1}. Then every weak solution $u \in \mathbb{U}_0$ to \eqref{eq:Sd} belongs 
        to $C(\overline{Q})$ and there is a constant $C>0$ such that
        \begin{equation}\label{eq:ucont}
            \|u\|_{C(\overline{Q})} \le C \|z\|_{L^r((0,T);L^p(\Om))} . 
        \end{equation}
    \end{corollary}
The continuity of $u$ in the elliptic case has been recently established in \cite{HAntil_DVerma_MWarma_2019b}. However, such a result for the parabolic problem \eqref{eq:Sd} is new to this work.

Our next result shows the well-posedness of the PDE when the right-hand side is a Radon measure.  This result is needed because the Lagrange multiplier corresponding to the state constraints \eqref{eq:Ud} is a measure.  Therefore, when we study the adjoint equation, we have a measure in the right-hand side of the equation.  We also allow for the possibility of non-zero measure-valued initial condition as can be seen in the following: For $z\in \mathcal{M}(\overline{Q})$ let $z_Q := z\vert_{Q}$ and $z_0 := z\vert_{\{0\} \times \overline{\Omega}}$, and consider 
\begin{align} \label{eq:SD_init_meas}
\begin{cases}
\partial_t u + (-\Delta)^s u 
&= z_{Q}  \quad \mbox{in } Q ,  \\
u  &=0 \quad\;\; \mbox{in } \Sigma ,  \\
u(0,\cdot) &= z_{0} \quad \mbox{in } \overline{\Om}. 
\end{cases} 
\end{align}

 \begin{theorem} \label{thm:zM} 
    Let $p,r$ fulfill \eqref{cond-p1} and $z \in \mathcal{M}(\overline{Q})$. Then there is a unique very weak solution $u \in 
        \left(L^r((0,T);L^p(\Om))\right)^\star$ of \eqref{eq:SD_init_meas}. Moreover, there is 
         a constant $C > 0$ such that 
        \begin{equation}\label{eq:LqLpbnd}
            \|u\|_{L^r((0,T);L^p(\Om))^\star} \le C \|z\|_{\mathcal{M}(\overline{Q})} . 
        \end{equation}
    \end{theorem}  
    
    The above result is used to show the well-posedness of the adjoint equation, a fractional parabolic PDE with measure-valued right-hand-side.  The adjoint equation arises when we derive the first order optimality conditions of the control problem \eqref{eq:dcp}.  We again emphasize the novelty of \eqref{eq:dcp} in that our PDE constraint is a fractional parabolic PDE with control taken in $Q$ (space and time) and we impose additional control and state constraints.  The Lagrange multiplier associated with the state constraints is a Radon measure, which necessitates the introduction of the notion of very-weak solutions to the adjoint equation (see Sections \ref{s:statecont} and \ref{s:ocp}).  Due to the difficulties in solving such an equation in practice, we introduce a regularized version of the optimal control problem using the Moreau-Yosida regularization (see \cite{KIto_KKunisch_2008a,HiHi2009, antil2019moreau}, among others). This enables us to create an algorithm to 
solve the optimal control problem.  While this approach is well-known, there are almost no works in the literature for fractional parabolic problems.  For example, see \cite{NeTr2009}, where the regularization is presented for parabolic problems in the classical case, but the optimality conditions of the original problem are not discussed.  For fractional problems, the Moreau-Yosida regularization was used recently in \cite{antil2019moreau} in the fractional elliptic case, and in this work we provide the details needed to extend those results to the current context.

The rest of the paper is organized as follows.
In Section~\ref{s:not}, we first introduce some 
notation and the relevant function spaces. The content of this section is well-known. 
The proofs of our main results stated above, as well as the precise definitions needed to state the results, are found in Section~\ref{s:statecont}.
In Section~\ref{s:ocp}, we study our optimal control problem, establish its well-posedness, and derive the first order optimality conditions.  We introduce the Moreau-Yosida regularized problem in Section~\ref{s:rocp} and show convergence
(with rate) of the regularized solutions to the original solution.  Our numerical examples in 
Section~\ref{s:num}, clearly confirm all our theoretical findings.

\section{Notation and preliminaries}
\label{s:not}

The goal of this section is to introduce some notation and state some preliminary results that are needed 
in the proofs of our main results. Unless otherwise stated, $\Om \subset \RR^N$$(N \ge 1)$ is an open 
bounded set, $0 < s < 1$. 
We define the space  
    \[
        W^{s,2}(\Om) := \Big\{ u \in L^2(\Om) \ : \int_{\Omega}\int_{\Omega}\frac{|u(x)-u(y)|^2}{|x-y|^{N+2s}}\;dxdy<\infty\Big\}, 
    \]
which is a Sobolev space when we endow it with the norm  
 \[
    \|u\|_{W^{s,2}(\Om)} := \left(\int_\Om |u|^2\;dx 
        + \int_{\Omega}\int_{\Omega}\frac{|u(x)-u(y)|^2}{|x-y|^{N+2s}}\;dxdy \right)^{\frac12}.      
 \]
We denote
\begin{equation*}
 W_{0}^{s,2}(\Omega ):=\overline{\mathcal{D}(\Omega )}^{W^{s,2}(\Omega )} ,
\end{equation*}
where we are taking $\mathcal{D}(\Om)$ to be the space of smooth functions with compact support in $\Omega$.

To study  the Dirichlet problem \eqref{eq:Sd} we also need to consider the following fractional order Sobolev space 
    \[
        \widetilde{W}^{s,2}_0(\Om) := \left\{ u \in W^{s,2}(\RR^N) \;:\; u = 0 \mbox{ in } 
                \RR^N\setminus\Om \right\} . 
    \]    
Notice that 
    \begin{align*}
        \|u\|_{\widetilde{W}_0^{s,2}(\Om)}:=\left(\int_{\RR^N}\int_{\RR^N}\frac{|u(x)-u(y)|^2}{|x-y|^{N+2s}}\;dxdy\right)^{\frac 12}
    \end{align*}
defines an equivalent norm on $\widetilde{W}_0^{s,2}(\Om)$.     

    \begin{remark}\label{remark}
    {\rm
        We recall the following result taken from \cite[Remark~2.2]{HAntil_DVerma_MWarma_2019b}.
        \begin{enumerate}[(i)]
            \item The embeddings 
            \begin{equation}\label{inj1}
                W_0^{s,2}(\Omega), \;\widetilde{W}^{s,2}_0(\Om) \hookrightarrow
                \begin{cases}
                    L^{2^\star}(\Omega)\;\;&\mbox{ if }\; N>2s,\\
                    L^p(\Omega),\;\;p\in[1,\infty)\;\;&\mbox{ if }\; N=2s,\\
                    C^{0,s-\frac{N}{2}}(\bOm)\;\;&\mbox{ if }\; N<2s,
                \end{cases}
            \end{equation}
            are continuous, where we have set
            \begin{align}\label{eq:2star}
                2^\star:=\frac{2N}{N-2s},\quad N>2>2s \mbox{ or if } N = 1 \mbox{ and } 0 < s < \frac12 .
            \end{align}

            \item Assume that that $\Omega$ has  a Lipschitz continuous boundary.  If $0<s\ne \frac 12<1$, then  the spaces  $W_0^{s,2}(\Omega)$ and $\widetilde{W}^{s,2}_0(\Om)$ coincide with equivalent norm. But if $s=\frac 12$, then $\widetilde{W}^{\frac 12,2}_0(\Om)$ is a proper subspace of $W_0^{\frac 12,2}(\Omega)$. We refer to \cite[Chapter 1]{Gris} for more details.       
        \end{enumerate}         
    }    
    \end{remark}

We shall denote by $
\widetilde{W}^{-s,2 }(\Om)$  the dual of 
$\widetilde{W}_{0}^{s,p}(\Om )$, i.e., 
$ \widetilde{W}^{-s,2}(\Om ):=(\widetilde{W}_{0}^{s,2}(\Om ))^{\star }$, with respect to the pivot space $L^2(\Omega)$ so that we have the following continuous embeddings:
\begin{align*}
 \widetilde{W}^{s,2}_0(\Om) \hookrightarrow L^2(\Omega)\hookrightarrow \widetilde{W}^{-s,2 }(\Om).
\end{align*}
For more details on fractional order Sobolev spaces we refer to \cite{NPV,Gris,War} and their references.

We are now ready to give a rigorous definition of the fractional Laplace operator $(-\Delta)^s$. Consider the space
    \begin{equation*}
        \mathbb{L}_s^{1}(\RR^N):=\left\{u:\RR^N\rightarrow
        \mathbb{R}\;\mbox{
        measurable}:\;\int_{\RR^N}\frac{|u(x)|}{(1+|x|)^{N+2s}}%
        \;dx<\infty \right\}.
    \end{equation*}
Then for a function $u$ in $\mathbb{L}_s^{1}(\RR^N)$ and $\varepsilon >0$, we let
    \begin{equation*}
        (-\Delta )_{\varepsilon }^{s}u(x)=C_{N,s}\int_{\{y\in \RR^N,|y-x|>\varepsilon \}}
        \frac{u(x)-u(y)}{|x-y|^{N+2s}}dy,\;\;x\in\RR^N . 
    \end{equation*}%
Here the normalization constant $C_{N,s}$ is given by 
    \begin{equation*}
        C_{N,s}:=\frac{s2^{2s}\Gamma\left(\frac{2s+N}{2}\right)}{\pi^{\frac
        N2}\Gamma(1-s)},
    \end{equation*}%
where $\Gamma $ denotes the standard Euler Gamma function (see, e.g. \cite{Caf3,NPV,War-DN1,War}). 
The fractional Laplacian $(-\Delta )^{s}$ is then defined for $u\in \mathbb{L}_s^{1}(\RR^N)$  by 
the formula
\[
(-\Delta )^{s}u(x)=C_{N,s}\mbox{P.V.}\int_{\RR^N}\frac{u(x)-u(y)}{|x-y|^{N+2s}}dy 
=\lim_{\varepsilon \downarrow 0}(-\Delta )_{\varepsilon
}^{s}u(x),\;\;x\in\RR^N, 
\]
provided that the limit exists for a.e. $x\in\RR^N$. 

Next we define the operator $(-\Delta)_D^s$ in $L^2(\Omega)$ as follows:
\[
D((-\Delta)_D^s):=\Big\{u|_{\Om},\; u\in \widetilde{W}_0^{s,2}(\Om):\; (-\Delta)^su\in L^2(\Omega)\Big\},\;\;(-\Delta)_D^s(u|_{\Om}):=(-\Delta)^su\;\mbox{ a.e. in }\;\Omega.
\] 
Notice that $(-\Delta)_D^s$ is the realization of the fractional Laplace operator $(-\Delta)^s$ in $L^2(\Om)$ 
with the zero Dirichlet exterior condition $u=0$ in $\RR^N\setminus\Omega$. The next result is well-known (see e.g. \cite{BWZ,ClWa,SV2}).  
\begin{proposition}\label{semigroup}
The operator $(-\Delta)_D^s$ has a compact resolvent and $-(-\Delta)_D^s$ generates  a strongly continuous submarkovian semigroup $(e^{-t(-\Delta)_D^s})_{t\ge 0}$ on $L^2(\Omega)$.
\end{proposition}

We refer to \cite{ClWa} for further qualitative properties of the semigroup $(e^{-t(-\Delta)_D^s})_{t\ge 0}$.
We conclude this section with the following observation.

\begin{remark}\label{rem-dsg}
{\em
The operator $(-\Delta)_D^s$ can be viewed as a bounded operator from $ \widetilde{W}_0^{s,2}(\Om)$ into $ \widetilde{W}^{-s,2}(\Om)$. In addition, we have that the operator  $-(-\Delta)_D^s$ generates a strongly continuous semigroup $(T(t))_{t\ge 0}$ on the space $ \widetilde{W}^{-s,2}(\Om)$. The semigroups $(T(t))_{t\ge 0}$ and $(e^{-t(-\Delta)_D^s})_{t\ge 0}$ coincide on $L^2(\Omega)$. Throughout the following, if there is no confusion we shall simply denote the semigroup $(T(t))_{t\ge 0}$ by $(e^{-t(-\Delta)_D^s})_{t\ge 0}$.
}
\end{remark}

\section{State equation and an equation with measure valued datum}
\label{s:statecont}

For the remainder of this section, unless otherwise stated, we assume that $\Omega\subset\RR^N$ ($N\ge 1$) is a bounded open set. For each result we shall clarify if a regularity on $\Omega$ is needed.
Moreover, given a Banach space $X$ and its dual $X^\star$, 
we shall denote by $\langle \cdot,\cdot\rangle_{X^\star,X}$ their duality pairing. 

The goal of this 
section is to establish the continuity of solutions to \eqref{eq:Sd} if the datum $z$ belongs to $L^r((0,T);L^p(\Om))$ 
with $p,r$ fulfilling \eqref{cond-p1} and to study the well-posdendess of \eqref{eq:Sd}
with $z \in (C(\overline{Q}))^\star = \mathcal{M}(\overline{Q})$. The latter space $\mathcal{M}(\overline{Q})$ 
denotes the space of all Radon measures on $\overline{Q}$ such that 
    \[
        \langle \mu , v \rangle_{(C(\overline{Q}))^\star,C(\overline{Q})} 
        = \int_{\overline{Q}} v\;d\mu , 
        \quad \mu \in \mathcal{M}(\overline{Q}), \quad v \in C(\overline{Q}) .
    \]
In addition, we have the following norm on this space:
    \[
        \|\mu\|_{\mathcal{M}(\overline{Q})} = \sup_{v \in C(\overline{Q}), \; |v|\le 1} \int_{\overline{Q}} v\;d\mu . 
    \] 

Towards this end, we notice that \eqref{eq:Sd} can be rewritten as the following Cauchy problem:
\begin{equation}\label{CP}
\begin{cases}
\partial_t u+(-\Delta)_D^s u=z\quad &\mbox{ in } Q,\\
u(0,\cdot)=0&\mbox{ in }\;\Omega.
\end{cases}
\end{equation}

Next we state the notion of weak solution to \eqref{CP}. 
 \begin{definition}[\bf Weak solution] 
 \label{def:weak_d}
    Let $z \in L^2((0,T);\widetilde{W}^{-s,2}(\Om))$. 
    A function \\$u \in \mathbb{U}_0 := L^2((0,T);\widetilde{W}^{s,2}_0(\Om)) \cap H^1((0,T);\widetilde{W}^{-s,2}(\Om))$ 
    is said to be a weak solution to \eqref{eq:Sd} if  $u(0,\cdot)=0$ a.e. in $\Omega$ and the equality
    \[
     \langle \partial_t u , v \rangle_{\widetilde{W}^{-s,2}(\Om),\widetilde{W}^{s,2}_0(\Om)} +
     \mathcal{E}(u,v)
         = \langle z , v \rangle_{\widetilde{W}^{-s,2}(\Om),\widetilde{W}^{s,2}_0(\Om)} , 
    \]
    holds, for every $v \in \widetilde{W}^{s,2}_0(\Om)$ and almost every $t \in (0,T)$. 
    Here 
    \begin{equation}\label{eq:E}
        \mathcal{E}(u,v) :=\frac{C_{N,s}}{2} 
        \int_{\RR^N}\int_{\RR^{N}} 
         \frac{(u(x)-u(y))(v(x)-v(y))}{|x-y|^{N+2s}} \;dxdy .
    \end{equation}
 \end{definition}
 
The existence of a weak solution to \eqref{eq:Sd} can be shown by using standard arguments.
In addition this weak solution belongs to $C([0,T];L^2(\Om))$.
From \cite[Proposition~3.3]{HAntil_DVerma_MWarma_2020a} we also recall that this weak solution $u$ can 
be written as follows.
 \begin{proposition}[\bf weak solution to \eqref{eq:Sd}]
 \label{prop:weak_Dir}
  Let  $z \in L^2((0,T);\widetilde{W}^{-s,2}(\Om))$. Then there exists a unique weak 
  solution $u \in \mathbb{U}_0$ to \eqref{eq:Sd} in the sense of 
  Definition~\ref{def:weak_d} and is given by
  \begin{align}\label{eq:uform}
  u(t,x)=\int_0^t e^{-(t-\tau)(-\Delta)_D^s}z(\tau,x)\;d\tau,
  \end{align}
  where $(e^{-t(-\Delta)_D^s})_{t\ge 0}$ is the semigroup mentioned in Remark \ref{rem-dsg}.
  In addition there is a constant $C>0$ such that
  \begin{align}\label{Es-DS-0}
   \|u\|_{\mathbb{U}_0} 
    \le C \|z\|_{L^2((0,T);\widetilde{W}^{-s,2}(\Om))}  .
  \end{align}
\end{proposition}

We are now ready to prove Theorem \ref{thm:bdd}. 

    
    \begin{proof}[\bf Proof of Theorem \ref{thm:bdd}] 
Firstly, if $1\le p<\frac{N}{2s}$, then the result is a direct consequence of the semigroup property, the representation \eqref{eq:uform} of solutions and the Sobolev embedding \eqref{inj1}. Thus, we can assume without any restriction that $p\ge \frac{N}{2s}$.
We give the proof for $p> \frac{N}{2s}$. The case $p= \frac{N}{2s}$ follows similarly as the case $p> \frac{N}{2s}$.

Let then $p$ and $r$ fulfill \eqref{cond-p1} and assume that $p> \frac{N}{2s}$.
We notice that it follows from the embedding \eqref{inj1} that the submarkovian semigroup  $(e^{-t(-\Delta)_D^s})_{t\ge 0}$ is ultracontractive in the sense that it maps $L^1(\Omega)$ into $L^\infty(\Omega)$. More precisely, following line by line the proof of \cite[Theorem 2.16]{GW-CPDE}, or the proof of the abstract result in \cite[Lemma 6.5]{Ouha} (see also \cite[Chapter 2]{Dav}), we get that  for every $1\le p\le q\le\infty$, there exists a constant $C>0$ such that for every $f\in L^p(\Omega)$ and $t>0$ we have the estimate:
\begin{align}\label{Ultra-cont}
\|e^{-t(-\Delta)_D^s}f\|_{L^q(\Omega)}\le C e^{-\lambda_1\left(\frac 1p-\frac 1q\right)t}t^{-\frac{N}{2s}\left(\frac 1p-\frac 1q\right)}\|f\|_{L^p(\Omega)},
\end{align}
where $\lambda_1>0$ denotes the first eigenvalue of the operator $(-\Delta)_D^s$.

Next, applying \eqref{Ultra-cont} with $q=\infty$ and using the representation \eqref{eq:uform} of the solution $u$, we get that
\begin{align}\label{Ultra-cont-2}
\|u(t,\cdot)\|_{L^\infty(\Omega)}\le C\int_0^t e^{-\lambda_1\frac{(t-\tau)}p}(t-\tau)^{-\frac{N}{2sp}}\|z(\tau,\cdot)\|_{L^p(\Omega)} d\tau.
\end{align}
Using Young's convolution inequality, we get from \eqref{Ultra-cont-2} that
\begin{align}\label{Ultra-cont-3}
\|u\|_{L^\infty(Q)}\le C\|z\|_{L^r((0,T),L^p(\Omega))}\left(\int_0^T   e^{-\lambda_1\frac{rt}{p(r-1)}}t^{-\frac{Nr}{2sp(r-1)}}\;dt\right)^{\frac{r-1}{r}}.
\end{align}
If $1>\frac{N}{2sp}\frac{r}{r-1}$, that is, if $p$ and $r$ fulfill \eqref{cond-p1}, then the integral in the right hand side of \eqref{Ultra-cont-3} is convergent. The proof is finished.
    \end{proof}    

Under the assumption that $\Omega$ has the exterior cone condition, we obtain Corollary \ref{cor:ucont} as a direct consequence of Theorem \ref{thm:bdd}.

    
    \begin{proof}[\bf Proof of Corollary \ref{cor:ucont}]
        We prove the result in two steps. 
        
        {\bf Step 1}: Let $\lambda\ge 0$ be real number, $f\in L^2(\Omega)$ and consider the following elliptic Dirichlet problem:
        \begin{equation}\label{EDP}
        \begin{cases}
        (-\Delta)^su+\lambda u=f\;\;\,&\mbox{ in }\;\Omega\\
        u=0 &\mbox{ in }\;\mathbb R^N\setminus\Omega.
        \end{cases}
        \end{equation}
  By a weak solution of \eqref{EDP} we mean a function $u\in \widetilde W_0^{s,2}(\Omega)$ such that  the equality
  \begin{align*}
  \frac{C_{N,s}}{2}\int_{\mathbb R^N}\int_{\mathbb R^N}\frac{(u(x)-u(y))(v(x)-v(y))}{|x-y|^{N+2s}}\;dxdy+\lambda \int_{\Omega}uv\;dx=\int_{\Omega}fv\;dx,
  \end{align*}
  holds for every $v\in \widetilde W_0^{s,2}(\Omega)$.
  
  The existence and uniqueness of weak solutions to the Dirichlet problem \eqref{EDP} are a direct consequence of the classical Lax-Milgram theorem.  
  
In our recent work \cite{HAntil_DVerma_MWarma_2019b}, we have shown that, if $f\in L^p(\Omega)$ with $p> \frac{N}{2s}$ (see also \cite{RS-DP} for the case $p=\infty$), then every weak solution $u$ of the Dirichlet problem \eqref{EDP} belongs to $C_0(\Omega):=\{u\in C(\bOm):\; u=0\;\mbox{ on }\;\pOm\}$. Thus, the resolvent operator $R(\lambda,(-\Delta)_D^s)$ maps $L^p(\Omega)$ into $C_0(\Omega)$ for every $\lambda\ge 0$.
  In particular, this shows that for every $t>0$, the operator $e^{-t(-\Delta)_D^s}$ maps $L^p(\Omega)$ ($p>\frac{N}{2s}$) into the space $C_0(\Omega)$, that is, the semigroup $(e^{-t(-\Delta)_D^s})_{t\ge 0}$ has the strong Feller property. In addition, we have the following result which is interesting in its own, independently of the application given in this proof.

Let $(-\Delta)_{D,c}^s$ be the part of the operator $(-\Delta)_D^s$ in $C_0(\Omega)$, that is,
\begin{equation*}
\begin{cases}
D((-\Delta)_{D,c}^s):=\Big\{u\in D((-\Delta)_{D}^s)\cap C_0(\Omega):\; ((-\Delta)_{D}^su)|_{\Omega}\in C_0(\Omega)\Big\},\\
 (-\Delta)_{D,c}^su= ((-\Delta)_{D}^su)|_{\Omega}.
\end{cases}
\end{equation*}

 From the above properties of the resolvent operator and the semigroup, together with the fact that $D((-\Delta)_{D,c}^s)$ is dense in $C_0(\Omega)$, we can deduce that the operator $-(-\Delta)_{D,c}^s$ generates a strongly continuous semigroup $(e^{-t(-\Delta)_{D,c}^s})_{t\ge 0}$ on $C_0(\Omega)$. Thus, for every $f\in C_0(\Omega)$ we have that the function
 \begin{align*}
 u(t,x):=\int_0^te^{-(t-\tau)(-\Delta)_{D,c}^s}f(x)\;d\tau
 \end{align*}
 belongs to $C([0,\infty),C_0(\Omega))$. We can then deduce that 
 for every $z\in C([0,T];C_0(\Omega))$, the unique weak solution $u \in \mathbb{U}_0$ to \eqref{eq:Sd} given  by 
 \begin{align*}
 u(t,x):=\int_0^te^{-(t-\tau)(-\Delta)_{D,c}^s}z(\tau,x)\;d\tau
 \end{align*}
 belongs  to $C(\overline{Q})$.
        
{\bf Step 2}: Now, let $p$ and $r$ satisfy \eqref{cond-p1}.
        Since the space $C([0,T];C_0(\Om))$ is dense in $L^r((0,T);L^p(\Om))$, 
         we can construct a sequence $\{z_n\}_{n \in \mathbb{N}}$ such 
        that 
        \[
            \{z_n\}_{n \in \mathbb{N}} \subset C([0,T];C_0(\Om)) \quad \mbox{ and } \quad 
            z_n \rightarrow z \mbox{ in } L^r((0,T);L^p(\Om))
                \mbox{ as } n\rightarrow \infty . 
        \]
        Let $u_n \in \mathbb{U}_0$ be the weak solution to \eqref{eq:Sd} with datum $z_n$.                 
        It follows from Step 1 that $u_n \in C(\overline{Q})$. 
        Whence, subtracting the equations satisfied by $(u,z)$ and $(u_n,z_n)$ and using the 
        estimate \eqref{eq:Linfbnd} from Theorem~\ref{thm:bdd} we obtain that there is a constant $C>0$ such that for every $n\in\NN$ we have
        \[
            \|u-u_n\|_{L^\infty(Q)} \le C \|z-z_n\|_{L^r((0,T);L^p(\Om))} . 
        \]        
        As a result, we have that $u_n \rightarrow u$ in $L^\infty(Q)$ as $n\rightarrow \infty$. Since
        $u$ is the uniform limit on $Q$ of a sequence of continuous functions $\{u_n\}_{n\in \mathbb{N}}$ on $\overline{Q}$, it follows that
        $u$ is also continuous on $\overline{Q}$. The proof is finished. 
    \end{proof}

Now, throughout the rest of the paper, without any mention, we assume that $\Omega\subset\RR^N$ ($N\ge 1$) is a bounded open set with a Lipschitz continuous boundary and has the exterior ball condition. 

Next we study the well-posedness of the state equation \eqref{eq:Sd} 
where $z$ is a Radon measure and the initial condition $u(0,\cdot)$ is equal to a Radon measure. 
However, prior to this result, we need to introduce the notion of such solutions. We call them very-weak solutions. 
We refer to \cite{HAntil_DVerma_MWarma_2019b} for a similar notion for the fractional
elliptic problems and \cite{HAntil_RKhatri_MWarma_2019a,HAntil_DVerma_MWarma_2020a} for
fractional elliptic and parabolic problems with nonzero exterior conditions.  To motivate the need to discuss 
such solutions, we introduce the adjoint equation, which we will derive in the next section when discussing 
the first-order optimality conditions.  For $\mu \in \mathcal{M}(\overline{Q})$, we can write $\mu=\mu_{Q}+\mu_{\Sigma}+\mu_{T}+\mu_{0}$, where $\mu_{Q}=\mu\vert_{Q}$, $\mu_{\Sigma}=\mu\vert_{\Sigma}$, $\mu_{T}=\mu\vert_{\{T\} \times \overline{\Omega}}$ and $\mu_{0}=\mu\vert_{\{0\} \times \overline{\Omega}}$.
We then consider the following problem with $\mu_0 = 0$ and $\mu_\Sigma = 0$
\begin{align} \label{eq:adj}
	\begin{cases}
	-\partial_t \xi + (-\Delta)^s \xi 
	&= \mu_{Q}  \quad  \mbox{in } Q ,  \\
	\xi &=0 \quad  \mbox{in } \Sigma ,  \\
	\xi(T,\cdot) &= \mu_{T} \quad \mbox{in } \overline{\Om}. 
	\end{cases} 
\end{align}    
Note that at the final time the data for the adjoint variable is a measure on $\{T\} \times \overline{\Omega}$.  Now, making a change of variables $t\mapsto T-t$ in \eqref{eq:adj} yields, for $\tilde{\xi}(t,\cdot)=\xi(T-t,\cdot)$,
\begin{align} \label{eq:mod_adj}
	\begin{cases}
	\partial_t \tilde{\xi} + (-\Delta)^s \tilde{\xi} 
	&= \mu_{Q}  \quad \mbox{in } Q ,  \\
	\tilde{\xi} &=0 \quad \mbox{in } \Sigma ,  \\
	\tilde{\xi}(0,\cdot) &= \mu_{T} \quad \mbox{in } \overline{\Om}.
	\end{cases}  
\end{align}
This problem now resembles \eqref{eq:SD_init_meas}, and so we define the notion of very-weak solution in this context. 
    \begin{definition}[\bf very-weak solutions]
    \label{def:vwsoln}
        Let $z \in \mathcal{M}(\overline{Q})$ and  $p,r$ satisfy \eqref{cond-p1}. A function $u \in \Big(L^r((0,T);L^p(\Om))\Big)^\star$ 
        is said to be a very-weak solution
        to \eqref{eq:SD_init_meas} if the identity 
        \begin{alignat*}{3}
            \int_Q u \left(-\partial_t v + (-\Delta)^s v\right)\;dxdt 
             &= \int_{\overline Q} v dz\\
             &= \int_Q v dz_Q(t,x) + \int_{\overline{\Om}} v(0,x) dz_0(x) ,
        \end{alignat*}
        holds for every 
        $v \in \Big\{ C(\overline{Q}) \cap 
        \mathbb{U}_0 \; : \;
        v(T,\cdot) = 0 \mbox{ a.e. in } \Om, \ \left(-\partial_t + (-\Delta)^s \right) v 
            \in L^r((0,T);L^p(\Om))  \Big\}$. 
    \end{definition}
Now, we are ready to prove the existence and uniqueness of very-weak solutions as stated in Theorem \ref{thm:zM}. 

    
    \begin{proof}[\bf Proof of Theorem \ref{thm:zM}] We prove the theorem in three steps.
        
        {\bf Step 1}: Given $\zeta \in L^r((0,T);L^p(\Om))$ where $p,r$ fulfill \eqref{cond-p1}, 
        we begin by considering the following ``dual" problem 
       \begin{equation}\label{eq:Sd_dual}
          \begin{cases}
            -\partial_t w + (-\Delta)^s w &= \zeta \quad \mbox{in } Q, \\
                w &= 0           \quad \mbox{in } \Sigma ,  \\
                w(T,\cdot) & = 0 \quad \mbox{in } \Om . 
         \end{cases}                
       \end{equation} 
    After using semigroup theory as in Proposition~\ref{prop:weak_Dir}, we can deduce that \eqref{eq:Sd_dual} 
    has a unique weak solution $w \in \mathbb{U}_0$. In addition, from \eqref{eq:Sd_dual}
    we have that $\left(-\partial_t + (-\Delta)^s \right) w \in L^r((0,T);L^p(\Om))$.     
    It follows from Corollary~\ref{cor:ucont} that $w \in C(\overline{Q})$. 
    Thus $w$ is a valid ``test function" according to Definition~\ref{def:vwsoln}. 
    
   {\bf Step 2}: Towards this end, we define the map
    \[
    \begin{aligned}
        \Xi : L^r((0,T);L^p(\Om)) &\rightarrow C(\overline{Q}) \\
                                \zeta &\mapsto \Xi \zeta =: w .
    \end{aligned}    
    \]
    Due to Corollary~\ref{cor:ucont}, $\Xi$ is linear and continuous. 
    
    We are now ready to construct a unique $u$. We set $u := \Xi^* z$, then 
    $u \in \left(L^r((0,T);L^p(\Om))\right)^\star$, moreover $u$ solves \eqref{eq:Sd}
    according to Definition~\ref{def:vwsoln}. Indeed
    \begin{equation}\label{eq:bb}
        \int_Q u \zeta \;dxdt = \int_Q u \left(-\partial_t w + (-\Delta)^s w \right) \;dxdt
        = \int_Q (\Xi^*z) \zeta \;dxdt = \int_{\overline{Q}} w dz, 
    \end{equation}
    that is, $u$ is a solution of \eqref{eq:Sd} according to Definition~\ref{def:vwsoln} and we have shown the existence. 
    
    Next, we prove the uniqueness. Assume that \eqref{eq:Sd} has two very weak solutions $u_1$ and $u_2$ with the same right hand side datum $z$. Then it follows from \eqref{eq:bb} that 
     \begin{equation}\label{aa}
    \int_Q (u_1-u_2) \left(-\partial_t v + (-\Delta)^s v \right) \;dxdt
        = 0,
    \end{equation}
    for every $v \in \Big\{ C(\overline{Q}) \cap 
        \mathbb{U}_0 \; : \;
        v(T,\cdot) = 0 \mbox{ a.e. in } \Om, \ \left(-\partial_t + (-\Delta)^s \right) v 
            \in L^r((0,T);L^p(\Om))  \Big\}$. 
It follows from Step 1 that the mapping 
\[
	\mathbb U_0\cap C(\overline Q) \to L^r((0,T);L^p(\Om)):\;\; v\mapsto \left(-\partial_tv+(-\Delta)^sv\right)
\] 
is surjective. Thus, we can deduce from \eqref{aa} that
 \begin{align*}
  \int_Q (u_1-u_2) w \;dxdt
        = 0,     
\end{align*}  
for every $w\in L^r((0,T);L^p(\Om))$.
  Exploiting the fundamental lemma of the calculus of variations we can conclude from the preceding identity that $u_1-u_2=0$ a.e. in $\Omega$ and we have shown the uniqueness.
    
 {\bf Step 3}: It then remains to show the bound \eqref{eq:LqLpbnd}. It follows from \eqref{eq:bb} that 
    \begin{equation}\label{eq:bbbb}
        \left| \int_Q u \zeta \;dxdt \right| 
         \le \|z\|_{\mathcal{M}(\overline{Q})} \|w\|_{C(\overline{Q})}
         \le C \|z\|_{\mathcal{M}(\overline{Q})} \|\zeta\|_{L^r((0,T);L^p(\Om))} ,
    \end{equation}
    where in the last step we have used Corollary~\ref{cor:ucont}. Finally dividing both sides of the estimate \eqref{eq:bbbb} by 
    $\|\zeta\|_{L^r((0,T);L^p(\Om))}$ and taking the supremum over all functions
    $\zeta \in L^r((0,T);L^p(\Om))$ we obtain the desired result. The proof is complete. 
    \end{proof}

\section{Optimal control problem}
\label{s:ocp}

The main goal of this section is to establish well-posdeness of the optimal control problem 
\eqref{eq:dcp} and to derive the first order necessary optimality conditions. 
We start by equivalently rewriting the optimal control problem \eqref{eq:dcp} in terms of the constraints
\eqref{CP}. Recall that $(-\Delta)^s_D$ is the realization of $(-\Delta)^s$ in $L^2(\Om)$ with zero
Dirichlet exterior conditions and it is a self-adjoint operator. In terms of $(-\Delta)^s_D$, \eqref{eq:dcp} becomes 
    	\begin{equation} 
	\begin{aligned}
		&\min_{(u,z)\in (U,Z)} J(u,z) \\ 
		\text{subject to}& \\
		& \partial_t u + (-\Delta)_D^s u = z , \quad \mbox{in } Q \\
		&u(0,\cdot) = 0 \quad \mbox{in } \Om \\
		&u|_{Q} \in \mathcal{K} \quad \mbox{and} \quad 
		z \in Z_{ad}  .
	\end{aligned}	
	\end{equation}

We next define the appropriate function spaces. Let 
    \begin{equation*}
    \begin{aligned}
        Z &:= L^r((0,T);L^p(\Om)) ,  \quad \mbox{with $p,r$ as in \eqref{cond-p1}  
         but } 1 < p < \infty, \; 1 < r < \infty , \\
        U &:= 
          \{u\in \mathbb{U}_0 \cap C(\overline{Q}): (\partial_t + (-\Delta)^s_D) (u|_\Om) \in L^r((0,T);L^{p}(\Om))\}.
    \end{aligned}    
    \end{equation*}    
Here, $U$ is a Banach space with the graph norm 
\[
\|u\|_{U} := \|u\|_{\mathbb{U}_0}+\|u\|_{C(\overline{Q})} + \|(\partial_t + (-\Delta)^s_D) (u|_\Om)\|_{L^r((0,T);L^{p}(\Om))}.
\] 
We let $Z_{ad} \subset Z$ to be a nonempty, closed, and convex set and $\mathcal{K}$ as in \eqref{eq:Ud}.  We require the spaces $U$ and $Z$ to be reflexive.  This is needed to show the existence of solution to $\eqref{eq:dcp}$.

Next using Corollary~\ref{cor:ucont} we have that for every $z \in Z$ there is a unique $u \in U$ that solves
\eqref{eq:Sd}. As a result, the following control-to-state map 
    \[
        S : Z \rightarrow U , \quad z \mapsto Sz =: u 
    \]
is well-defined, linear, and continuous. Due to the continuous embedding of $U$ in $C(\overline{Q})$
we can in fact consider the control-to-state map as 
    \[
        E \circ S : Z \rightarrow C(\overline{Q}) ,
    \]
and we can define the admissible control set as 
    \[
        \widehat{Z}_{ad} := \left\{ z \in Z \; : \; z \in Z_{ad}, \ (E\circ S)z \in \mathcal{K} \right\} ,
    \]
and thus the reduced minimization problem is given by 
    \begin{equation}\label{eq:rocp}
        \min_{z \in \widehat{Z}_{ad}} \mathcal{J}(z):= J((E\circ S)z,z) . 
    \end{equation}        

Towards this end, we are ready to state the well-posedness of \eqref{eq:rocp} and equivalently
\eqref{eq:dcp}. 
    \begin{theorem}
        Let $Z_{ad}$ be a closed, convex, bounded subset of $Z$ and $\mathcal{K}$ a closed 
        and convex subset of $C(\overline{Q})$ such that $\widehat{Z}_{ad}$ is nonempty. 
        Moreover, let $J : L^2(Q) \times L^r((0,T);L^p(\Om)) 
        \rightarrow \RR$ be weakly lower-semicontinuous. Then \eqref{eq:rocp} has a solution. 
    \end{theorem}
    \begin{proof}
        The proof follows by using similar arguments as in the elliptic case, 
        see \cite[Theorem~4.1]{HAntil_DVerma_MWarma_2019b} and has been omitted for brevity. 
    \end{proof}

Next, we derive the first order necessary conditions under the following Slater condition.
    \begin{assumption}
        There is some control $\widehat{z} \in Z_{ad}$ such that the corresponding state $u$ fulfills 
        the strict state constraints
        \begin{equation}\label{eq:slater}
            u(t,x) < u_b(t,x) , \quad \forall (t,x) \in \overline{Q} . 
        \end{equation}
    \end{assumption}

Under this assumption, we have the following first order necessary optimality conditions.

    \begin{theorem}\label{thm:optcond}
        Let $J : L^2(Q) \times L^r((0,T);L^p(\Om)) \rightarrow \RR$ be continuously Fr\'echet differentiable
        and let \eqref{eq:slater} hold. Let $(\bar{u},\bar{z})$ be a solution to the optimization
        problem \eqref{eq:dcp}. Then there are Lagrange multipliers 
        $\bar\mu \in \mathcal{M}(\overline{Q})$ and $\bar\xi \in \left(L^r((0,T);L^p(\Om))\right)^\star$ such that
        \begin{subequations}
        \begin{align}
            &\partial_t \bar u + (-\Delta)^s_D \bar u = \bar{z} , \quad \mbox{in } Q, 
             \quad \bar u(0,\cdot) = 0, \mbox{ in } \Om  , \label{eq:a}  \\             
            &\begin{cases}
                -\partial_t \bar{\xi} + (-\Delta)^s_D \bar{\xi} 
                    &= J_u(\bar{u},\bar{z}) + {\bar\mu_Q} , \quad \mbox{in } Q , \\
                 \bar\xi(\cdot,T) &= {\bar\mu_T}, \quad \mbox{in } \overline{\Om} \label{eq:b},    
            \end{cases}    \\
             &   \langle \bar{\xi} + J_z(\bar{u},\bar{z}) , 
                    z - \bar{z} \rangle_{{L^r((0,T));L^{p}(\Om))^\star,L^r((0,T));L^{p}(\Om))}} 
                 \ge 0 , && \forall \;z \in Z_{ad}  \label{eq:c}  \\
              &\bar{\mu} \ge 0, \quad  \bar{u}(x) \le u_b(x) \mbox{ in } \overline{Q}, 
             \quad \mbox{and} \quad \int_{\overline{Q}} (u_b - \bar{u})\;d\bar\mu = 0 \label{eq:d} ,     
        \end{align} 
        where \eqref{eq:a} and \eqref{eq:b} are understood in the weak sense (see Definition~\ref{def:weak_d}) and very-weak sense (see Definition~\ref{def:vwsoln}), respectively. 
        \end{subequations}
    \end{theorem}
    \begin{proof}
        We will check the requirements of \cite[Lemma~1.14]{MHinze_RPinnau_MUlbrich_SUlbrich_2009a} to 
        complete the proof. Notice that $\partial_t + (-\Delta)^s_D : U \mapsto Z$ is bounded and 
        surjective. In addition, we have that the interior of $\mathcal{K}$ is nonempty due to 
        \eqref{eq:slater}. In order to apply \cite[Lemma~1.14]{MHinze_RPinnau_MUlbrich_SUlbrich_2009a}, 
        the only thing that remains to be shown is the existence of 
        $(\hat{u},\hat{z}) \in U \times Z$ such that
        \begin{equation}\label{eq:1}
            \partial_t (\hat{u}-\bar{u}) + 
            (-\Delta)^s_D (\hat{u}-\bar{u}) - (\hat{z}-\bar{z}) = 0  \;\;\mbox{ in }\;Q, \quad (\hat{u}-\bar{u}) = 0 \;\;\mbox{in }\; \Om .   
        \end{equation}
        Notice that for simplicity we have suppressed the initial condition. 
        We recall that $(\bar{u},\bar{z})$ solves the state equation, as a result we obtain that 
        \begin{equation}\label{eq:2}
            \partial_t \hat{u} + 
            (-\Delta)^s_D \hat{u} = \hat{z}  \;\;\mbox{ in }\;Q,
            \quad \hat{u} = 0 \;\;\mbox{in }\; \Om 
        \end{equation}
        Since for every $\hat{z} \in Z_{ad}$, there is a unique $\hat{u}$ that
        solves \eqref{eq:2}, in particular $(\hat{u},\hat{z})$ works. Then using 
        \cite[Lemma~1.14]{MHinze_RPinnau_MUlbrich_SUlbrich_2009a} we obtain 
        \eqref{eq:a}--\eqref{eq:c}. Moreover, \eqref{eq:d} follows from 
        \cite[Lemma~1.14]{MHinze_RPinnau_MUlbrich_SUlbrich_2009a} and the 
        discussions given in  \cite[Page 88]{MHinze_RPinnau_MUlbrich_SUlbrich_2009a}. 
        The proof is complete. 
    \end{proof}

\section{Moreau-Yosida Regularization of Optimal Control Problem}
\label{s:rocp}

The purpose of this section is to study the regularized optimal control problem using the well-known Moreau-Yosida regularization and to show that the regularized problem is an approximation to the original problem. We refer to \cite{antil2019moreau} for the elliptic case. From hereon we will assume that $u_b\equiv 0$, however all the results follow for the general case with slight modifications.
The Moreau-Yosida regularized optimal control problem is given by
\begin{subequations}\label{eq:regdcp}
	\begin{equation}\label{eq:regJd}
	\min J^\ga(u,z):=J(u,z)+ \frac{1}{2 \ga} \|(\hat{\mu}+\ga u)_+\|^2_{L^2(Q)},
	\end{equation}
	subject to the fractional parabolic PDE: Find $u \in U$ solving 
	\begin{equation}
	\begin{cases}
	\partial_t u + (-\Delta)^s_D u &= z \quad \mbox{in } Q, \\
	u(0,\cdot) &= 0 \quad \mbox{in } \Omega  ,
	\end{cases}                
	\end{equation} 
	with 
	\begin{equation}\label{eq:regZd}
	z \in Z_{ad}, 
	\end{equation}
\end{subequations}
where $0\leq \hat{\mu}\in L^2(Q)$ is the realization of the Lagrange multiplier $\bar{\mu}$ and $\ga >0$ denotes the regularization parameter. Here, $(\cdot)_+$ denotes max$\{0,\cdot\}$. More information about this can be found in \cite{KIto_KKunisch_2008a}.
From hereon, we will use a particular cost functional $J$, given by,
\[ J(u,z):=\frac{1}{2}\|u-u_d\|_{L^2{(Q)}}^2 + \frac{\alpha}{2} \|z\|^2_{L^2(Q)}, \] 
and we choose the relevant function spaces as,
\begin{equation*}
\begin{aligned}
Z &:= L^r((0,T);L^p(\Om)) ,  \quad \mbox{with $p,r$ as in \eqref{cond-p1}  
	but } 2 \le p < \infty, \; 2 \le r < \infty , \\
U &:= 
\{u\in \mathbb{U}_0 \cap C(\overline{Q}): (\partial_t + (-\Delta)^s_D) (u|_\Om) \in L^r((0,T);L^{p}(\Om))\}.
\end{aligned}    
\end{equation*}    
Then, again using the control-to-state mapping $S$, \eqref{eq:regdcp} can be reduced to
\begin{equation}\label{eq:regrpDir}
\min_{z \in Z_{ad}} \mathcal{J}^{\ga}(z) :=\mathcal{J}(z)+\frac{1}{2 \ga} \|(\hat{\mu}+\ga {S}z)_+\|^2_{L^2(Q)}. 
\end{equation}
Existence and uniqueness of solution to the problem \eqref{eq:regdcp} can be done using the
direct method, see for instance \cite[Theorem~4.1]{HAntil_DVerma_MWarma_2019b}. 

Moreover, we have the following first order necessary and sufficient optimality conditions. 

\begin{theorem}\label{thm:optcondReg}
	Let $J^\ga : L^2(Q) \times L^r((0,T);L^p(\Om)) \rightarrow \RR$ be continuously Fr\'echet differentiable
	and let $(\bar{u}^\ga,\bar{z}^\ga)$ be a solution to the regularized optimization
	problem \eqref{eq:regdcp}. Then there exists a Lagrange multiplier $\bar\xi ^\ga \in \mathbb{U}_0$ such that 
	\begin{subequations}
		\begin{align}
		&\begin{cases} \label{eq:rega}
		\partial_t \bar u ^\ga + (-\Delta)^s_D \bar u ^\ga & = \bar{z} ^\ga , \quad \mbox{in } Q,\\
		\bar u ^\ga(0,\cdot) &= 0, \quad \mbox{ in } \Om ,
		\end{cases} 
		 \\             
		&\begin{cases}\label{eq:regb}
		-\partial_t \bar{\xi} ^\ga+ (-\Delta)^s_D \bar{\xi} ^\ga 
		&= \bar{u} ^\ga - u_d + (\hat{\mu} +\ga \bar u ^\ga)_+ , \quad \mbox{in } Q , \\
		\bar\xi^\ga(T,\cdot) &= 0 \quad \mbox{in } \Om ,    
		\end{cases}    \\
		&   \langle \bar{\xi}^\ga + \alpha \bar{z} ^\ga ,z - \bar{z}^\ga \rangle_{L^2(Q))} 
		\ge 0 , \quad \forall \;z \in Z_{ad} \label{eq:regc} .
		\end{align}
	\end{subequations} 
\end{theorem}

\begin{proof}
	The proof is similar to the proof of \cite[Theorem~4.2]{HAntil_DVerma_MWarma_2019b} and has been omitted for brevity. 
\end{proof}

Next, let us begin the analysis to show that the regularized problem \eqref{eq:regdcp} is indeed an approximation to the original problem \eqref{eq:dcp}. We modify the approach of \cite{antil2019moreau} and apply it to the parabolic case. Let us start by deriving a uniform bound on the regularization term. Observe that for $\ga\ge 1$,
\begin{align} \label{eq:rel}
\begin{aligned}
\mathcal{J}(\bar{z}^\ga)\le \mathcal{J}^\ga(\bar{z}^\ga)\le \mathcal{J}^\ga(\bar{z})
&\le \mathcal{J}(\bar{z})+\frac{1}{2 \ga} \|\hat{\mu}\|^2_{L^2(Q)} \\
&\le \mathcal{J}(\bar{z})+\frac{1}{2} \|\hat{\mu}\|^2_{L^2(Q)} \eqqcolon C_{\bar{z}} .
\end{aligned}	
\end{align}

Using \eqref{eq:rel} we get that $\frac{1}{2 \ga} \|(\hat{\mu}+\ga \bar{u}^\gamma)_+\|^2_{L^2(Q)}$ is uniformly bounded. Also from \eqref{eq:rel} we have that 
\begin{align*}
\|(\bar{u}^\ga)_+\|^2_{L^2(Q)}
&\le \frac{2}{\ga} \left(\mathcal{J}(\bar{z})-\mathcal{J}(\bar{z}^\ga)+\frac{1}{2\ga} \|\hat{\mu}\|^2_{L^2(Q)}\right) .
\end{align*}
Thus 
\begin{align} \label{uplusbnd}
\|(\bar{u}^\ga)_+\|_{L^2(Q)}\le \om(\ga^{-1})\ga^{-\frac{1}{2}},
\end{align}
where 
\[ 
\om(\ga^{-1}) := 2\text{ max }\left(\frac{1}{2\ga} \|\hat{\mu}\|^2_{L^2(Q)},\left(\mathcal{J}(\bar{z})-\mathcal{J}(\bar{z}^\ga)\right)_+ \right)^{1/2}.
\]
Since $\bar{z}^\ga\rightarrow \bar{z}$ strongly in $L^2(Q)$ as $\gamma \to\infty$ (by \cite[Proposition~2.1]{hintermuller_kunisch}) and $\mathcal{J}$ is continuous, we obtain that $\om (\gamma^{-1})\downarrow0$ as $\gamma \to \infty$.
Moreover, from \eqref{eq:rel} we can deduce that $\mathcal{J}(\bar{z}^\ga)\le C_{\bar{z}}$ which yields
\begin{align} \label{regsolbound}
\text{max}\left( \|\bar{u}^\ga - u_d\|^2_{L^2(Q)}, \alpha \|\bar{z}^\ga\|^2_{L^2(Q)}  \right) \le 2 C_{\bar{z}}. 
\end{align}
Then from \eqref{eq:regb}, along with \eqref{uplusbnd} and \eqref{regsolbound},
we obtain that there is a constant $C>0$ independent of $\ga$ such that
\begin{align*} 
\|\bar{\xi}^\ga\|_{\mathbb{U}_0}
& \le C\left(\|\bar{u}^\ga - u_d\|_{L^2(Q)} + \|\hat{\mu}\|_{L^2(Q)}+ \gamma\|(\bar{u}^\ga)_+\|_{L^2(Q)} \right) \\
& \le C\left(2+\om(\ga^{-1})\sqrt{\ga}) \right).
\end{align*} 

We now estimate the distance between $(\bar{u},\bar{z})$ and $(\bar{u}^{\ga},\bar{z}^\ga)$. 

\begin{theorem} \label{distance}
	Let	$(\bar{u},\bar{z})$ and $(\bar{u}^{\ga},\bar{z}^\ga)$ denote the solutions of \eqref{eq:dcp} and \eqref{eq:regdcp}, respectively. Then,
	\begin{align} \label{eq:distance}
	\alpha \|\bar{z}-\bar{z}^\ga\|^2_{L^2(Q)} &+ \|\bar{u}-\bar{u}^\ga\|^2_{L^2(Q)}+\ga \|(\bar{u}^\ga)_+\|^2_{L^2(Q)} \nonumber \\
	&\le \frac{1}{\ga} \|\hat{\mu}\|^2_{L^2(Q)}+
	\left\langle \bar{\mu},\bar{u}^\ga \right\rangle _{\mathcal{M}(\overline{Q}),C(\overline{Q})},
	\end{align}
	and hence,
	\begin{align} \label{eq:violation}
	\|(\bar{u}^\ga)_+\|_{L^2(Q)}\le \sqrt{\frac{2}{\ga}} \max \left( \frac{1}{\ga}\|\hat{\mu}\|^2_{L^2(Q)},\left\langle \bar{\mu},(\bar{u}^\ga)_+ \right\rangle _{\mathcal{M}(\overline{Q}),C(\overline{Q})}\right)^{\frac 12}.
	\end{align}
\end{theorem}

\begin{proof}
	We star by using the optimality conditions \eqref{eq:c} and \eqref{eq:regc} with $z=\bar{z}^\ga$ and $z=\bar{z}$, respectively. This yields
	\begin{align}
	\alpha \|\bar{z}-\bar{z}^\ga\|^2_{L^2(Q)} 
	&\le \int_{Q} (\bar{z}-\bar{z}^\ga)(\bar{\xi}^\ga - \bar{\xi}) \;dxdt
	. \label{eq:4a}
	\end{align}		 
	Next, we use the state equations \eqref{eq:a} and \eqref{eq:rega} to get that, for every $v \in \left(L^r((0,T);L^p(\Omega))\right)^\star$,
	\begin{align}\label{eq:4b}
	\langle \partial_t(\bar{u} - \bar{u}^\ga) +& (-\Delta)^s_D(\bar{u} - \bar{u}^\ga), v \rangle_{L^r((0,T);L^p(\Omega)),\left(L^r((0,T);L^p(\Omega))\right)^\star} \notag\\
	=& \langle \bar{z} - \bar{z}^\ga, v \rangle_{L^r((0,T);L^p(\Omega)),\left(L^r((0,T);L^p(\Omega))\right)^\star} .
	\end{align}
	Recall that both $\bar{\xi} \in \left(L^r((0,T);L^p(\Omega))\right)^\star$ and $\bar{\xi}^\ga \in L^2((0,T);\widetilde{W}^{s,2}_0(\Omega)) \hookrightarrow \left(L^r((0,T);L^p(\Omega))\right)^\star$ due to the fact that $\left(L^r((0,T);L^p(\Omega))\right)^\star\cong L^{r'}((0,T);L^{p'}(\Omega))$ 
	(\cite[Theorem 1.3.10]{HyNeVeWe2016}, for example) and Remark~\ref{remark} (iii), so we can set $v := \bar{\xi}^\ga - \bar{\xi}$ in \eqref{eq:4b}. Subsequently, substituting \eqref{eq:4b} in \eqref{eq:4a}, and using \eqref{eq:b} and \eqref{eq:regb}, along with 
	$\bar{u}, \bar{u}^\gamma \in U$, we obtain 
	\begin{align}\label{eq:4c}
	\alpha \|\bar{z}-\bar{z}^\ga\|^2_{L^2(Q)} 
	&\le \langle \partial_t(\bar{u} - \bar{u}^\ga) + (-\Delta)^s_D(\bar{u} - u^\ga), \bar{\xi}^\ga - \bar{\xi} \rangle_{L^r((0,T);L^p(\Omega)),L^r((0,T);L^p(\Omega))^\star} \nonumber  \\
	&= -\|\bar{u} - \bar{u}^\ga \|_{L^2(Q)}^2 +\int_{Q} (\hat{\mu}+\ga \bar{u}^\ga)_+(\bar{u}-\bar{u}^\ga)\;dxdt - \left\langle \bar{\mu},\bar{u}-\bar{u}^\ga \right\rangle _{\mathcal{M}(\overline{Q}),C(\overline{Q})} \nonumber \\
	\implies \alpha \|\bar{z}-\bar{z}^\ga\|^2_{L^2(Q)} + &\|\bar{u} - \bar{u}^\ga \|_{L^2(Q)}^2 \le \left\langle \bar{\mu},\bar{u}^\ga \right\rangle _{\mathcal{M}(\overline{Q}),C(\overline{Q})} + 
	\int_{Q} (\hat{\mu}+\ga \bar{u}^\ga)_+(\bar{u}-\bar{u}^\ga) \;dxdt\nonumber \\
	&\hspace{2.37cm}\le \left\langle \bar{\mu},\bar{u}^\ga \right\rangle _{\mathcal{M}(\overline{Q}),C(\overline{Q})} + 
	\int_{Q} (\hat{\mu}+\ga \bar{u}^\ga)_+(-\bar{u}^\ga)\;dxdt  , 
	\end{align}
	where we have used that $\left\langle \bar{\mu},\bar{u} \right\rangle _{\mathcal{M}(\overline{Q}),C(\overline{Q})} = 0$ from \eqref{eq:d} and $\int_{Q} (\hat{\mu}+\ga \bar{u}^\ga)_+(\bar{u})\;dxdt \le 0$.\\
	Next, let 
	\begin{equation}\label{eq:4f}
	Q_{\ga}^+(\hat{\mu})\coloneqq \{(t,x)\in Q \; : \;\hat{\mu}+\ga \bar{u}^\ga >0\}.
	\end{equation}
	Then 
	$$ \int_{Q \setminus Q_{\ga}^+(\hat{\mu})} 
	(\hat{\mu}+\ga \bar{u}^\ga)_+(-\bar{u}^\ga) \;dxdt= 0.$$
	 This allows us to write 
	\begin{align}\label{eq:4e}
	\int_{Q} (\hat{\mu}+\ga \bar{u}^\ga)_+(-\bar{u}^\ga)\;dxdt
	&= \int_{Q_{\ga}^+(\hat{\mu})} (\hat{\mu}+\ga \bar{u}^\ga)(-\bar{u}^\ga)\;dxdt \nonumber \\
	&\le - \ga \|(\bar{u}^\ga)_+\|^2_{L^2(Q)}  
	+ \frac{1}{\ga} \|\hat{\mu}\|^2_{L^2(Q)},	 	 
	\end{align}
	where the last inequality follows from the fact that
	\begin{equation} \label{eq:uPlusNorm}
	\ga \|\bar{u}^\ga\|^2_{L^2(Q_{\ga}^+(\hat{\mu}))} \ge \ga \|\bar{u}^\ga\|^2_{L^2(Q_{\ga}^+(0))}=\ga \|(\bar{u}^\ga)_+\|^2_{L^2(Q)},
	\end{equation}
	because $Q_{\ga}^+(\hat{\mu})\supseteq Q_{\ga}^+(0):=\{(x,t)\in Q : \bar{u}^\ga>0\}$ and $\bar{u}^\ga> \dfrac{-\hat{\mu}}{\ga}$ in $Q_{\ga}^+(\hat{\mu})$.
	Finally, substituting \eqref{eq:4e} in \eqref{eq:4c}, we obtain \eqref{eq:distance}. 
	
	Next, we use that $(\bar{u}^\ga)_+\in C(\overline{Q})$. This follows from the fact that $\bar{u}^\ga$ is the solution to the state equation \eqref{eq:Sd}, hence, it belongs to $C(\overline{Q})$. This, together with the non negativity of $\bar{\mu}\in \mathcal{M}(\overline{Q})$ combined with \eqref{eq:distance}, yields \eqref{eq:violation}. The proof is complete.
\end{proof}

\begin{theorem} \label{udorder}
	If $0\in Z_{ad}$ and $u_d\ge 0$ a.e. in $Q$, then $\|(\bar{u}^\ga)_+\|_{L^2(Q)}=\mathcal{O} (\ga^{-1})$ as $\ga \rightarrow \infty$.
\end{theorem}

\begin{proof}
	We will start with the use of the regularized state and adjoint equations. The regularized state and the adjoint variables $\bar{u}^\ga$ and $\bar{\xi}^\ga$ satisfy, for every $v \in \widetilde{W}^{s,2}_0(\Om)$ and almost every $t \in (0,T)$,
	\begin{align*}
	( \partial_t \bar{u}^\ga , v )_{L^2(\Om)} +
	\mathcal{E}(\bar{u}^\ga,v)
	&= \int_{\Om} \bar{z}^\ga v\;dx  ,\\
	-( \partial_t \bar{\xi}^\ga , v )_{L^2(\Om)} +
	\mathcal{E}(\bar{\xi}^\ga,v)
	&=\int_{\Om} (\bar{u}^\ga - u_d)v\;dx+ \int_{\Om} (\hat{\mu}+\ga \bar{u}^\ga)_+ v \;dx.
	\end{align*}
	Substituting $v=\bar{\xi}^\ga$ in the first equation and $v=\bar{u}^\ga$ in the second, subtracting and integrating over time yields
	\begin{equation}\label{eq:5a}
	0=\|\bar{u}^\ga\|^2_{L^2(Q)}-\int_{Q} \bar{u}^\ga u_d \;dxdt+\int_{Q} (\hat{\mu} +\ga\bar{u}^\ga)_+ \bar{u}^\ga \;dxdt- \int_{Q} \bar{z}^\ga \bar{\xi}^\ga\;dxdt.
	\end{equation}
	Next, using the definition of the set $Q_\ga^+(\hat\mu)$ from \eqref{eq:4f} in conjunction with \eqref{eq:5a}, we obtain that 
	\begin{alignat*}{5}
	0&\ge \|\bar{u}^\ga\|^2_{L^2(Q)}-\int_{Q} \bar{u}^\ga u_d\;dxdt +\int_{Q_{\ga}^+(\hat{\mu})} \hat{\mu} \bar{u}^\ga\;dxdt + \int_{Q_{\ga}^+(\hat{\mu})} \ga (\bar{u}^\ga)^2 - \int_{Q} \bar{z}^\ga \bar{\xi}^\ga\;dxdt\\
	&\ge \|\bar{u}^\ga\|^2_{L^2(Q)}-\int_{Q} \bar{u}^\ga u_d\;dxdt -\int_{Q_{\ga}^+(\hat{\mu})} \ga^{-1} \hat{\mu}^2\;dxdt  +\int_{Q_{\ga}^+(0)} \ga (\bar{u}^\ga)^2\;dxdt - \int_{Q} \bar{z}^\ga \bar{\xi}^\ga\;dxdt,
	\end{alignat*}
	where in the last inequality we have used a similar technique as in the end of the previous proof. Then setting  $z:=\frac{\bar{z}^\ga}{2} \in Z_{ad}$ (because $Z_{ad}$ is convex and $0 \in Z_{ad}$) in \eqref{eq:regc}, we obtain that 
	\begin{alignat*}{3}
	0 &\geq   \|\bar{u}^\ga\|^2_{L^2(Q)}-\int_{Q} \bar{u}^\ga u_d \;dxdt-\int_{Q_{\ga}^+(\hat{\mu})} \ga^{-1} \hat{\mu}^2\;dxdt  + \ga \|(\bar{u}^\ga)_+\|_{L^2(Q)}^2 +\alpha \|\bar{z}^\ga\|_{L^2(Q)}^2,
	\end{alignat*}
	where we have once again used \eqref{eq:uPlusNorm}. The rest of the proof follows the lines of the proof of \cite[Theorem~5]{antil2019moreau}. The proof is finished.
\end{proof}

Now, we relax the condition on $u_d$ in Theorem \ref{udorder}. The proof is similar to \cite[Theorem~6]{antil2019moreau} and has been omitted for brevity. Recall that $(\cdot)_+=\text{max}\{0,\cdot\}$.

\begin{theorem} \label{relugaorder}
	If there exists $\varepsilon>0$ such that 
	\[ 
	-\int_{Q} u_d^- \bar{u}\;dxdt -\|\bar{u}\|^2_{L^2(Q)} -\alpha \|\bar{z}\|^2_{L^2(Q)} \le -\varepsilon,
	\]
	then $\|(\bar{u}^\ga)_+\|_{L^2(Q)}=\mathcal{O} (\ga^{-1})$ as $\ga \rightarrow \infty$. Here, $u_d^-$ denotes the negative part of the function $u_d$. 
\end{theorem}

\section{Numerical Experiments}
\label{s:num}

In this section we present numerical experiments to show the validity of Theorem \ref{udorder}. We set $\Omega \subset \mathbb R^2$ to be a disk of radius $1/2$ centered at the origin. We truncate $\mathbb R^2\backslash \Omega$ and bound the exterior of $\Omega$ with a circle of radius $3/2$ centered at the origin. For these experiments, we create a triangular mesh of $\Omega$ and its truncated exterior consisting of 1920 triangles.  In time we discretize the interval $[0,1]$ into equal sized subintervals of length 0.01. 

To solve the state and adjoint problem we use standard $\mathcal P_1$ Lagrangian finite elements in space and in time we use backward Euler time stepping.   For details on the discretization of the fractional Laplacian operator see \cite{acosta2017short}.  While we are using continuous piecewise linear functions to approximate $u$ and $p$, we use piecewise constants to approximate the control $z$, that is, $z$ is constant on each element of the mesh. To minimize the objective function, we use the BFGS method.  For the following experiments we only consider the state constraint $u_b$, and do not consider control constraints.  

To make what follows easier to read, we define 
\[
\widetilde{u}(x,y) = \frac{2^{-2s}}{\Gamma(1+s)^2}\left(1/4-(x^2 + y^2)_+\right)^s,
\]
and subsequently define our desired state and state constraint respectively to be 
\[
u_d (x,y,t) = 10t^2\widetilde{u}, \qquad 
	u_b(x,y,t) = \frac{1}{10} (1-t)^4 \widetilde{u}.
\]
For $s = 0.8$, Figure \ref{fig:1} shows the convergence of $\|(\bar{u}^\gamma - u_b)_+\|_{L^2(Q)}$ as $\gamma$ increases.  Notice that we get a convergence rate of $\mathcal O(\gamma^{-1})$, which agrees with our theoretical result in Theorem~\ref{udorder} as we are taking $u_d \geq 0$. 

\begin{figure}[ht]
\centering
\includegraphics[width=0.5\textwidth]{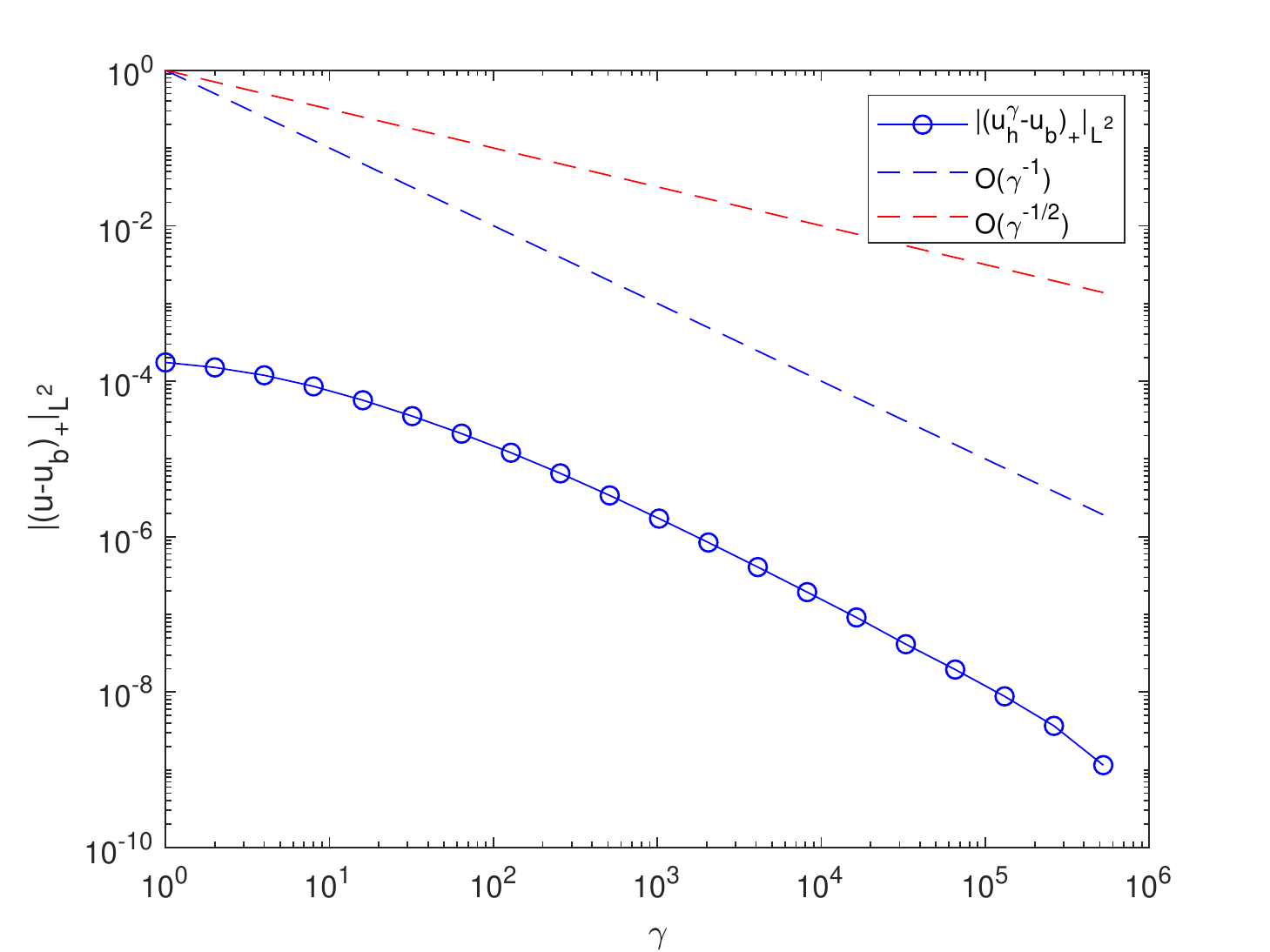} 
\caption{Convergence of $\|(\bar{u}^\gamma-u_b)_+\|_{L^2(Q)}$ as $\gamma$ increases for $s = 0.8$. We observe a linear rate of convergence in $\gamma$ as expected from Theorem~\ref{udorder}.  }\label{fig:1}
\end{figure} 

In Figures \ref{fig:3} and \ref{fig:4} we show some snapshots of numerical experiments in a slightly different context.  We keep $s = 0.8$, but now use
\[
u_d (x,y,t) := 10(1+t)\widetilde{u},  \qquad 
	u_b(x,y,t): = \frac14 \widetilde{u},
\] 
where we note that $u_b$ is constant in time.  We show $u_d$ at time $t = 0.88$ and $u_b$ in Figure \ref{fig:2}.  We also implement a constant force in the state equation, that is, our PDE constraint takes the form of 
\[
\partial_t u + (-\Delta)^s u = f + z,
\]
where we take $f = 1000$ to be constant in space and time. For $\gamma = 1,048,576$, we show the optimal state, control, adjoint, and Lagrange multiplier at $ t = 0.02$ in Figure \ref{fig:3} and at $t = 0.88$ in Figure \ref{fig:4}.  We note that inbetween these two time steps there is a period from $t = 0.03$ to $t = 0.86$ in which the state constraints are not active. In both plots, we clearly notice that the Lagrange multiplier is a measure. In addition, the adjoint variable is non-smooth, as expected.

\begin{figure}[ht]
\centering
\includegraphics[width=0.45\textwidth]{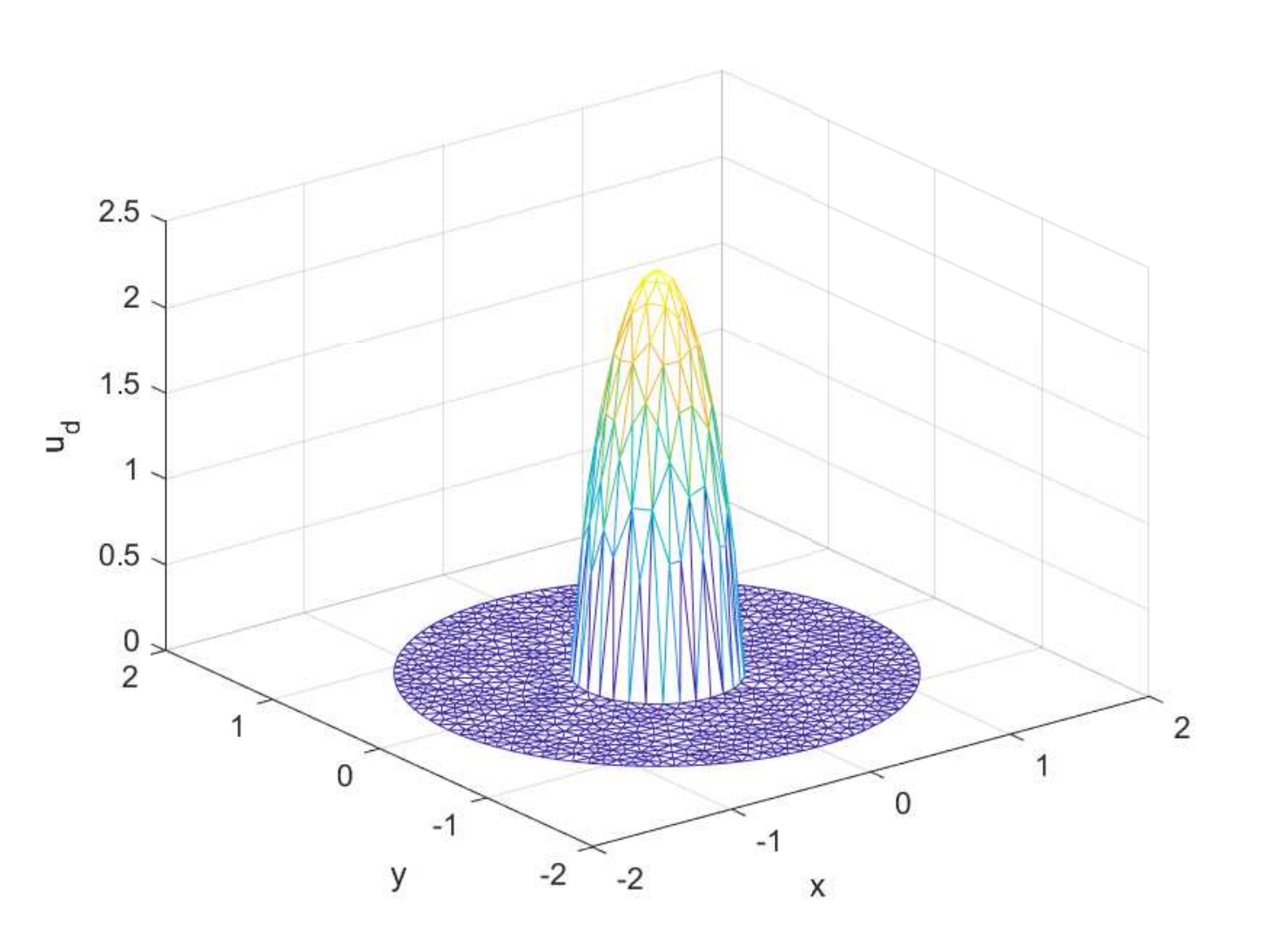} \qquad
\includegraphics[width = 0.45\textwidth]{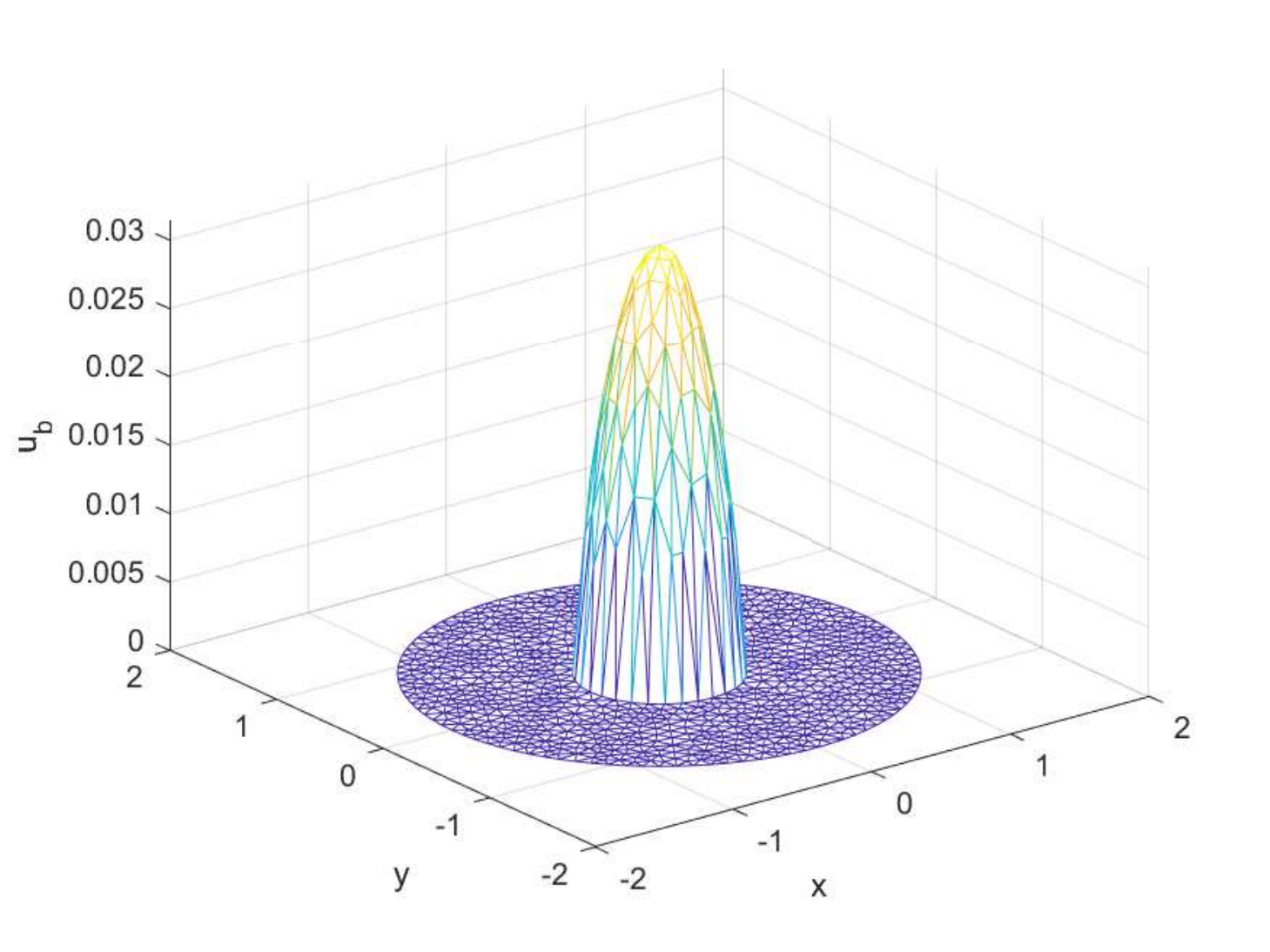} 
\caption{The desired state $u_d$ (left) and state constraint $u_b$ (right) for $s = 0.8$ at $t = 0.88$}\label{fig:2}
\end{figure}

\begin{figure}[ht]
\centering
\includegraphics[width=0.8\textwidth]{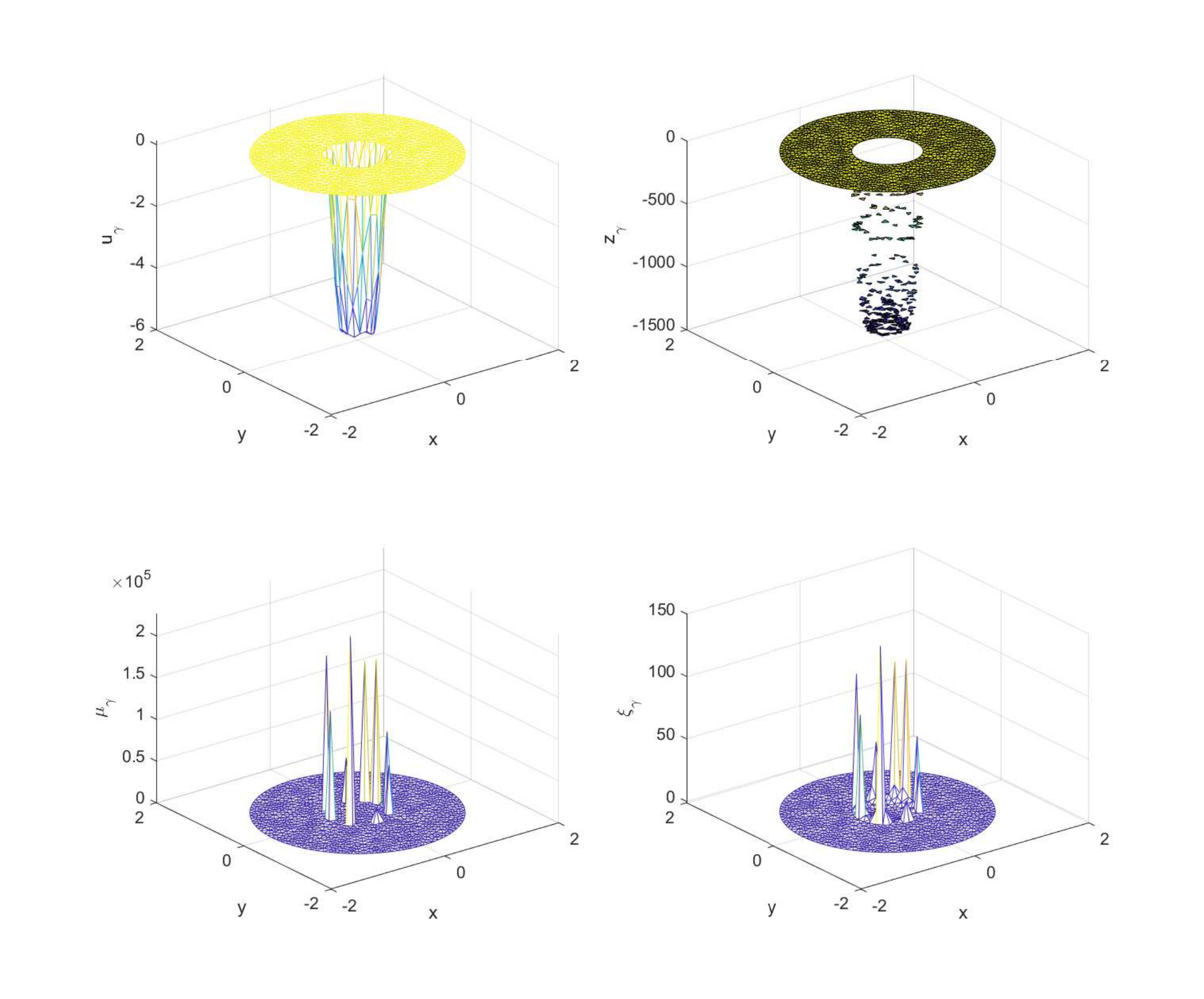}
\caption{The optimal state (upper left), optimal control (upper right), Lagrange multiplier (lower left), and optimal adjoint (lower right) for $s = 0.8$ and $\gamma = 1,048,576$ at time $t = 0.02$ \,.}\label{fig:3}
\end{figure}

\begin{figure}[ht]
\centering
\includegraphics[width=0.8\textwidth]{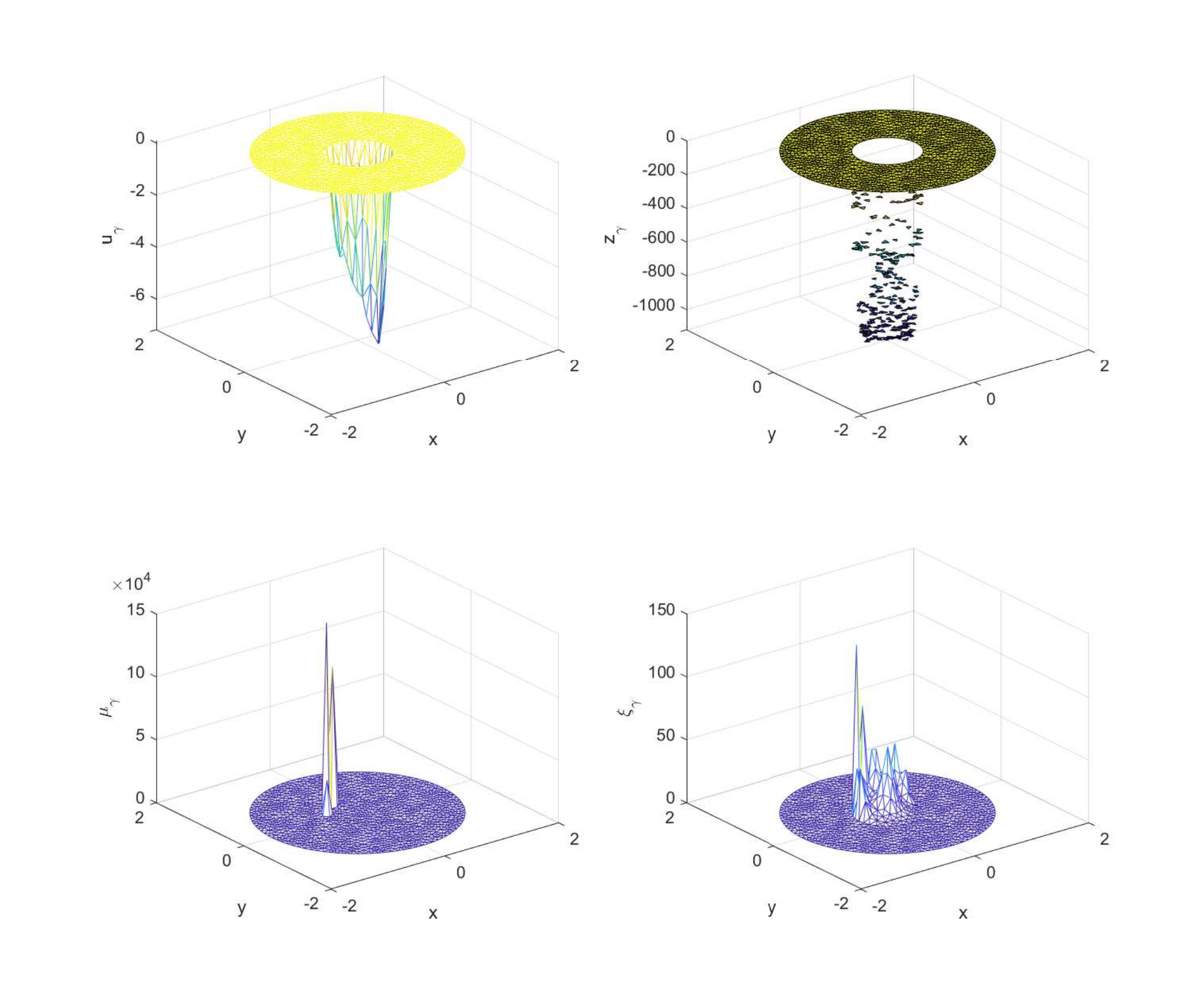}
\caption{The optimal state (upper left), optimal control (upper right), Lagrange multiplier (lower left), and optimal adjoint (lower right) for $s = 0.8$ and $\gamma = 1,048,576$ at time $t = 0.88$ \,.}\label{fig:4}
\end{figure}

\section{Conclusion and Open Questions}
In this work, we have presented and analyzed an optimal control problem with a fractional parabolic PDE constraint as well as control and state constraints.  The primary focus has been to show the well-posedness of this problem and establish some regularity on the state variable $u$.  Using the representation of $u$ from the theory of strongly continuous submarkovian semigroups, we have shown that $u$ is essentially bounded and continuous. We also introduced the reduced form of the control problem and derived the first order optimality conditions.  With a perspective towards implementation, we have introduced the Moreau-Yosida regularization of the control problem and showed the convergence of these regularized approximations to the solution of the original problem.  The theoretical rate of convergence has been verified numerically with an implementation using finite elements and backward Euler time-stepping. 


The focus of the paper is on the analysis and algorithm for the optimal control problem. But there remain several open questions, especially on the regularity of the original state equation. It remains open to establish the continuity of solutions to the fractional parabolic problem when the right-hand-side belongs to a dual space. Notice that this result is known in the elliptic case \cite{HAntil_DVerma_MWarma_2019b}. Another interesting direction to pursue is to carry out a detailed error analysis for the discrete problem.

\bibliographystyle{plain}
\bibliography{refs}

\end{document}